\documentclass[a4paper]{amsart}
\usepackage{amssymb}
\usepackage{array,multirow}
\usepackage{amsmath} %f{\"u}r Symbole der reellen Zahlen
\usepackage{pifont}
\usepackage[usenames]{color}
\usepackage{amsfonts}
\usepackage[all]{xypic}
\usepackage{graphicx}
\usepackage{psfrag}
\usepackage{verbatim}
\usepackage{pstricks}
\usepackage[latin1]{inputenc}
\usepackage{mathrsfs}
\usepackage[all,knot]{xy}
\usepackage{amssymb, amsthm, amsmath, pifont} %f{\"u}r Symbole der reellen Zahlen
\usepackage{enumerate}
\usepackage{pdfpages}

\newtheorem{thm}{Theorem}%[section]
\newtheorem{corollary}[thm]{Corollary}
\newtheorem{lemma}[thm]{Lemma}
\newtheorem{example}[thm]{Example}
\newtheorem{definition}[thm]{Definition}
\newtheorem{remark}[thm]{Remark}
\newtheorem{prop}[thm]{Proposition}
\newtheorem{claim}[thm]{Claim}
\newtheorem{observation}[thm]{Observation}

\newenvironment{Example}{\begin{example}\rm}{\end{example}}

\newenvironment{Remark}{\begin{remark}\rm}{\end{remark}}

\def\Z{\mathbb{Z}}

\def\epsilon{\varepsilon}

\def\s{{\sigma}}

\begin{document}
%\addtolength{\baselineskip}{+.3\baselineskip}
%\date{\today}
\address{Department of Mathematics, Boston College, Chestnut Hill, MA 02467, USA}
\author{Peter Feller}
\email{peter.feller.2@bc.edu}
\thanks{The author gratefully acknowledges support by the Swiss National Science Foundation Grant~155477.}
%: \emph{Positive Braids and Deformations}.}
%\urladdr{}
%\thanks{The author is supported by the Swiss National Science Foundation (project no.\ 137548).}
\title[Alexander polynomial bounds for the topological slice genus]{The degree of the Alexander polynomial is an upper bound for the topological slice genus}
\begin{abstract}%In this short note,
We use the famous knot-theoretic consequence of Freedman's disc theorem\textemdash knots with trivial Alexander polynomial bound a locally-flat disc in the $4$-ball\textemdash to prove the following generalization. The degree of the Alexander polynomial of a knot is an upper bound for twice its topological slice genus. We provide examples of knots where this determines the topological slice genus.
\end{abstract}
\maketitle
\section{Introduction}For a knot $K$\textemdash a smooth and oriented embedding of the unit circle $S^1$ into the unit $3$-sphere $S^3$\textemdash the \emph{topological slice genus} $g_4^{\rm top}(K)$ is the minimal genus of locally-flat, oriented
surfaces $S$ in the closed unit $4$-ball $B^4$ with oriented boundary $K\subset \partial B^4=S^3$. A celebrated theorem of Freedman asserts that every knot $K$ with trivial Alexander polynomial is \emph{topologically slice}, i.e.~$g_4^{\rm top}(K)$ equals $0$~\cite[Theorem~1.13]{Freedman_82_TheTopOfFour-dimensionalManifolds}; see also~\cite[11.7B]{FreedmanQuinn_90_TopOf4Manifolds} and~\cite[Appendix]{GaroufalidisTeichner_04_OnKnotswithtrivialAlex}. Here, the \emph{Alexander polynomial} $\Delta_K$, first introduced by Alexander~\cite{Alexander_28_TopInvsOfKnotsAndLinks},
%\in\Z[t,t^{-1}]$
is the Laurent polynomial with integer coefficients in the indeterminate $t$
defined by
\[\det\left(M\sqrt{t}-M^T\frac{1}{\sqrt{t}}\right),\]
where $M$ is any Seifert matrix for $K$ and $M^T$ is its transpose.
The \emph{degree} $\deg(\Delta_K)$ of the Alexander polynomial $\Delta_K$ is the difference of the largest and the smallest exponent among the exponents of the  monomials of $\Delta_K$; i.e.~$\deg(\Delta_K)$ is the breadth of~$\Delta_K$.
\begin{thm}\label{thm:alexupperboundtopslicegenus}
For every knot, the degree of its Alexander polynomial %$\Delta_K$\textemdash the maximal power of a monomial appearing in $\Delta_K$\textemdash
is greater than or equal to twice its topological slice genus.
\end{thm}
An appealing way of summarizing Theorem~\ref{thm:alexupperboundtopslicegenus} and the classical genus bound of the Alexander polynomial is the following. For all knots $K$, we have
\[2g_4^{\rm top}(K)\leq\deg(\Delta_K)\leq 2g(K)%\quad\text{for all knots $K$}
,\]
where $g(K)$ denotes the genus of $K$\textemdash the minimal genus of Seifert surfaces for $K$.
Theorem~\ref{thm:alexupperboundtopslicegenus} %is optimal, e.g.~the above inequalities are equlities on torus knots of braid $2$, and
determines $g_4^{\rm top}$ for many knots including some of small crossing number; examples are provided in Section~\ref{sec:app}.

%Theorem~\ref{thm:alexupperboundtopslicegenus} is a consequence of Freedman's result and the following statement:
The following purely 3-dimensional proposition reduces Theorem 1 to the genus zero case. The proof of this proposition, detailed in Section~\ref{sec:proof}, is completely elementary whereas the genus zero case uses the entire Freedman machine of infinite constructions in topological 4-manifold theory.

\begin{prop}\label{prop:splitting}
Let $K$ be a knot. Every Seifert surface $S$ of $K$ contains a separating simple closed curve $L$ with the following properties:
\begin{itemize}
\item The Alexander polynomial of $L$ (as a knot in $S^3$) is trivial.
\item The connected component $C$ of $S\setminus L$ that does not contain $K$ is a Seifert surface for $L$ with% first Betti number
\[2\mbox{\rm genus}\, (C)=2\mbox{\rm genus}\,(S)-\deg(\Delta_K).\]
\end{itemize}
\end{prop}

The reduction of Theorem~\ref{thm:alexupperboundtopslicegenus} to Proposition~\ref{prop:splitting} and Freedman's result is rather direct. The same idea was used by Rudolph to provide examples of torus knots for which the topological slice genus is smaller than their genus~\cite[Theorem~2]{Rudolph_84_SomeTopLocFlatSurf}.
\begin{proof}[Proof of Theorem~\ref{thm:alexupperboundtopslicegenus}]
For a given knot $K$, let $L\subset S$ be a simple closed curve in some Seifert surface $S$ of $K$ with the properties described in Proposition~\ref{prop:splitting}.
By removing the connected component $C$ of $S\setminus L$ that does not contain $K$, one obtains a surface $\overline{S\setminus C}\subset S^3$ with boundary $K\dot{\cup} L$. By Freedman's work~\cite[Theorem~1.13]{Freedman_82_TheTopOfFour-dimensionalManifolds}, $L$ bounds a topological locally-flat disc $D$ in $B^4$. Gluing $\overline{S\setminus C}$ and $D$ along $L$ yields a locally-flat surface $S^{\textrm{top}}$ of genus $\frac{\deg(\Delta_K)}{2}$ in $B^4$ with boundary $K\subset S^3=\partial B^4$. To be explicit, $S^{\textrm{top}}$ can be given as follows: shrink $D$ by a factor of $2$, yielding a disc in the $4$-ball $B^4_{\frac{1}{2}}$ of radius $\frac{1}{2}$ with boundary $L$ viewed as a knot in the $3$-sphere of radius $\frac{1}{2}$. Then, embed $S\setminus C$ in $B^4\setminus B^4_{\frac{1}{2}}= S^3\times(\frac{1}{2},1]$ via the map
\[S\setminus C\to S^3\times(\frac{1}{2},1],\quad x\mapsto \left(x, \frac{1}{2}+\frac{\mbox{dist}(L,x)}{2\mbox{dist}(K,x)+2\mbox{dist}(L,x)}\right),\]
and set $S^{\textrm{top}}\subset B^4$ to be the union of the shrunken $D$ and the embedded $S\setminus C$.
\end{proof}
We conclude the introduction by describing previous work relating $\deg(\Delta_K)$ and $g_4^{\rm top}$.
Borodzik and Friedl proved that $\deg(\Delta_K)+1$ is an upper bound for the algebraic unknotting number $u_a$, which follows from combining~\cite[Lemma~2.3]{BorodzikFriedl_15_TheUnknottingnumberAndClassInv1} and~\cite[Theorem~1.1]{BorodzikFriedl_14_OnTheAlgUnknottingNr}.
%\footnote{Caution: Borodzik and Friedl's notion of $\deg(\Delta_K)$ differs from ours by a factor of two.}
Since $g_4^{\rm top}\leq u_a$, this yields that
$g_4^{\rm top}\leq \deg(\Delta_K)+1$. %, which appears to be the best bound of $g_4^{\rm top}$ in terms of $\deg(\Delta_K)$ prior to Theorem~\ref{thm:alexupperboundtopslicegenus}.
\begin{comment}
Previously, the best known upper bound for $g_4^{\rm top}$ in terms of $\deg(\Delta_K)$ appears to have been
\begin{equation}\label{eq:BF}
g_4^{\rm top}\leq u_a \leq 2 \deg(\Delta_K)+1, \end{equation}
where $u_a$ denotes the algebraic unknotting number. The second inequality in~\eqref{eq:BF}
follows by combining Borodzik and Friedl's bound on their $n$-invariant~\cite[Lemma~2.3]{BorodzikFriedl_15_TheUnknottingnumberAndClassInv1} and their result that $n=u_a$~\cite[Theorem~1.1]{BorodzikFriedl_13_OnTheAlgUnknottingNr}.\footnote{Caution:
Borodzik and Friedl's notion of $\deg{\Delta_K}$ differs from ours by a factor of two.}%: ith their convention degree equals breadth.} %The fact that~\eqref{eq:BF} yields a bound which is
\end{comment}

{\bf{Acknowledgements}:}
The present article grew out of questions that arose during a joint effort with Sebastian Baader, Lukas Lewark, and Livio Liechti to determine the topological slice genus for families of torus knots and positive knots. I thank them for their helpful remarks; in particular, I thank Lukas for improving the statement of Proposition~\ref{prop:splitting}. Thanks also to Josh Greene for valuable inputs and Stefan Friedl for pointing me to work on the algebraic unknotting number. Finally, I thank the referee for helpful suggestions.
\section{Applications}\label{sec:app}
%As an application of Theorem~\ref{thm:alexupperboundtopslicegenus}, we determine the topological slice genus of a few knots for which the topological slice genus is unknown according to the knot census [knot info], see Section~\ref{sec:examples}.
Combining Theorem~\ref{thm:alexupperboundtopslicegenus} with classical bounds for the topological slice genus, e.g.~Kauffman and Taylor's signature bound~\cite[Theorem~3.13]{KauffmanTaylor_76_SignatureOfLinks}, yields simple criteria to determine the topological slice genus. Indeed, let $\sigma(K)$ denote the signature of a knot $K$ as introduced by Trotter~\cite{Trotter_62_HomologywithApptoKnotTheory}, i.e.~the signature $\s(M+M^T)$ of the symmetrization $M+M^T$ of any Seifert matrix $M$ for the knot $K$.
\begin{corollary}\label{cor:sign=deg}
For every knot $K$, we have
\[\left|\sigma(K)\right|
%\overset{\text{\cite{KauffmanTaylor_76_SignatureOfLinks}}}
{\leq} 2g_4^{\rm top}(K)
%\overset{\text{Theorem~\ref{thm:alexupperboundtopslicegenus}}}
{\leq} \deg(\Delta_K).\]
In particular, if $\left|\sigma(K)\right|=\deg(\Delta_K)$, then $g_4^{\rm top}(K)$ equals $\left|\frac{\sigma(K)}{2}\right|=\frac{\deg(\Delta_K)}{2}$.\qed
\end{corollary}
\begin{Example}
 We describe an infinite family of knots for which %Corollary~\ref{cor:sign=deg} determines %knots to have arbitrarily large $g_4^{\rm top}$ and where
 $g_4^{\rm top}$ is arbitrarily large, while being arbitrarily smaller than the smooth slice genus $g_4$.

For any positive integer $g$, any integer $2g\times 2g$-matrix $M$ for which $M-M^T$ has determinant $1$ describes the Seifert form on a Seifert surface $S$ bounded by some knot $K$; in fact, $S$ can be chosen to be a quasipositive Seifert surface as proven by Rudolph~\cite{Rudolph_83_ConstructionsOfQP1}\cite[Theorem 1.2]{Rudolph_89_QPandnewKnotInv}. %For a symmetric matrix $N$, let $\s(N)$ denote its signature.
If one chooses $M$ to satisfy \[\left|{\s(M+M^T)}\right|=\deg\left(\det\left(M\sqrt{t}-M^T\frac{1}{\sqrt{t}}\right)\right)<2g,\]
then one has examples of knots $K$ for which
\[\deg(\Delta_K)=2g_4^{\rm top}(K)<2g=2g_4(K),\]
by Corollary~\ref{cor:sign=deg} and the fact that quasipositive surfaces realize the smooth slice genus; see Rudolph's slice-Bennequin inequality~\cite{rudolph_QPasObstruction}.
Of course, the above examples include knots for which $g_4^{\rm top}$ is determined by Freedman's result; e.g.~if $K$ has trivial Alexander polynomial, or if $K$ is a connected sum of knots with trivial Alexander polynomial and knots for which $g=|\frac{\s}{2}|$.
However, for most knots $K$ as above, we do not know of a method that determines the topological slice genus and that does not use Theorem~\ref{thm:alexupperboundtopslicegenus}.
\end{Example}
Next, we apply Corollary~\ref{cor:sign=deg} to knots with small crossing number:
\begin{Example}
We determine the topological slice genus of the following knots, which can be represented by diagrams with 12 crossings.
One has ${\s}=-\deg(\Delta_K)=-4$ for the two knots $12n830$ and $12n750$,
${\s}=\deg(\Delta_K)=2$ for the two knots $12n519$ and $12n411$,
and ${\s}=-\deg(\Delta_K)=-2$ for $12n321$ and $12n293$; where designations are as in \emph{KnotInfo}~\cite{knotinfo}.
Previously, the topological slice genus appears to have been unknown for all these knots; compare~\cite{knotinfo}. We remark that, for $12n830$ and $12n750$, the smooth slice genus is known to be $3$, and for $12n321$ and $12n293$ it is known to be 2; in particular, it is strictly larger than the topological slice genus; while for $12n519$ and $12n411$ the smooth slice genus appears to be unknown (it is either 1 or 2). In Section~\ref{sec:12n750}, we discuss the knot $12n750$ explicitly.
\end{Example}

\section{Proof of Proposition~\ref{prop:splitting}}\label{sec:proof}
We provide a sketch of our proof of
Proposition~\ref{prop:splitting}.
For a given knot $K$,
we fix (in all of Section~\ref{sec:proof})
a Seifert surface $S$ %For a given knot $K$
and denote its genus by $g$.
%the rank of $H_1(S,\Z)$ by $r=2g$, where, of course, $g$ is the genus of $S$.
We find a basis for $H_1(S,\Z)$ for which the corresponding Seifert matrix $M$ is of the following form:
$M-M^T$ is the standard symplectic form on $\Z^{2g}$ and the bottom right corner of $M$ is a square matrix $N$ of size
${2g-\deg(\Delta_K)}$ which represents the Seifert form of a knot with trivial Alexander polynomial. %a special form given in~\eqref{eq:M}.
Then we represent this basis by simple closed curves such that for all pairs of curves the geometric intersection number equals the algebraic intersection number and choose a curve $L$ that separates the curves that represent the last $2g-\deg(\Delta_K)$ elements of this basis. Thus, $N$ is a Seifert matrix for $L$ and, therefore, $L$ has trivial Alexander polynomial.
We note that,
if $\Delta_K=1$, then $L$ is parallel to $K$ and this proof essentially reduces to the proof of%~\cite[Lemma 2]{} and
~\cite[Lemma~4.2]{GaroufalidisTeichner_04_OnKnotswithtrivialAlex}; see Remark~\ref{Rem:Obs:M}.

 In order to provide a detailed proof of Proposition~\ref{prop:splitting}, we
recall some facts about Seifert matrices and bilinear forms. %The Seifert form is a bilinear form on $H_1(S,\Z)$.
By choosing
a basis for $H_1(S,\Z)$, i.e.~by identifying $H_1(S,\Z)$ with $\Z^{2g}$, the Seifert form becomes a bilinear form on $\Z^{2g}$, which is canonically identified with a ${2g}\times2g$-matrix $M$\textemdash a \emph{Seifert matrix}.
%A Seifert matrix $M$ is given as the matrix representing the Seifert form on $\Z^{2g}by identifying $H_1(S,\Z)$ with $\Z^{2g}$ writing the Seifert form with respect to this basis.
The skew-symmetrization
$M-M^T$ of $M$ represents the \emph{intersection form} $I$ on $H_1(S,\Z)$ (with respect to the same basis) and, therefore, has determinant $1$.
A change of basis
%on $H_1(S,\Z)$
amounts to changing $M$ to $A^TMA$ for some $\Z$-invertible $2g\times 2g$-matrix $A$ (which amounts to performing a finite number of elementary column operations and their corresponding elementary row operations on $M$).
Fix a skew-symmetric bilinear form $F$ %with determinant $1$
on a finitely generated free abelian group $V$%of rank $r$
, e.g.~$I$ on $H_1(S,\Z)$.
A basis for $V$ is called \emph{symplectic} (with respect to $F$), if the corresponding matrix representing $F$, % with respect to this basis,
e.g.~$M-M^T$,
is the standard symplectic form %on $\Z^{r}$, i.e.~
\[\left[\begin{array}{cc}
0&-1\\ 1&0
\end{array}\right]\oplus\cdots\oplus\left[\begin{array}{cc}
0&-1\\ 1&0
\end{array}\right].\]
A necessary and sufficient condition for the existence of a symplectic basis for $V$ is that $F$ is invertible, i.e.~has determinant $1$ when identified with a matrix.

\begin{proof}[Proof of Proposition~\ref{prop:splitting}]
First, we make an observation from linear algebra:
 \begin{lemma}\label{lemma:M} %We first show that
 There exists a basis $B$ for $H_1(S,\Z)$ and a non-negative integer $d$ such that the Seifert form on $H_1(S,\Z)$ with respect to $B$ is given by a
 $2g\times 2g$-matrix of the following form:
%Assume that $d<g(S)$ as otherwise the statement is trivial. We first show that by a base change of $H_1(S,\Z)$, we may assume that
\begin{equation}\label{eq:M}\left[
\begin{tabular}{lcc|cc}{\begin{tabular}{ccc|cc}&&&&0\\
&$M_{2d}$&&$v_{g-d}$&$\vdots$\\
&&&&0\\\hline
&$v_{g-d}^T$&&0&0%$\epsilon_{g-d}$
\\
0&$\cdots$&0&1%$\delta_{g-d}$
&0\\
\end{tabular}}&&
&\begin{tabular}{c}
\\
\\
\\
$v_{1}$\\
\end{tabular}&\begin{tabular}{c}
0\\
\\
\\
$\vdots$\\
\\
\end{tabular}\\
&$\ddots$&\\
&&&&0\\
\hline
$\quad\quad\quad\quad\quad\quad  v_{1}^T$&&&0&0%$\epsilon_{1}$
\\
0$\quad0\quad\quad\quad\quad\cdots$&0&0&1%$\delta_{1}$
&0\\
\end{tabular}\right],\end{equation}
where $M_{2d}$ is a $2d\times 2d$-matrix with non-zero determinant and $v_i$ are column vectors with $2g-2i$ entries. Furthermore,
the degree of $\Delta_K$ equals $2d$.
%, and $(\epsilon_i,\delta_i)\in\{(0,1);(1,0)\}$ for all $i$.
\end{lemma}
\begin{proof} Let $M_{2g}$ be a Seifert matrix representing the Seifert form on $H_1(S,\Z)$ with respect to some basis. %Recall that performing base changes on $H_1(S,\Z)$ amounts to changing $M_{2g}$ to $A^TM_{2g}A$ for some $\Z$-invertible $r\times r$-matrix $A$(which amounts to performing a finite number of elementary column operations and their corresponding elementary row operation on $M_{2g}$).
We consider the case $\det(M_{2g})=0$ %(which is equivalent to $d<r$)
as otherwise the statement is trivial. Thus, by a change of basis, %basis $(e_1,\cdots,e_r)$ for $H_1(S,\Z)$ where $e_2$ pairs to zero with every other vector,
we can arrange that the last column of $M_{2g}$ consists of $0$'s only (this is done by choosing a primitive vector in the kernel of $M_{2g}$ and extending it to a basis). From $\det(M_{2g}-M_{2g}^T)=1$%(this holds for every Seifert matrix of a knot)
, we deduce that the greatest common divisor of the entries of the last row of $M_{2g}-M_{2g}^T$, which is equal to the last row of $M_{2g}$, is $1$. Therefore, we can change basis again (by performing elementary column operations on $M_{2g}$ simulating the Euclidean algorithm that yields the greatest common divisor of the last row) %$(e_1,\cdots,e_{2g})$ of $H_1(S,\Z)$
such that the corresponding Seifert matrix takes the form
\[\begin{tabular}{ccc|cc}&&&&0\\
&$M_{{2g}-2}$&&$w_{1}$&$\vdots$\\
&&&&0\\\hline
&$v_{1}^T$&&x&0%$\epsilon_{g-d}$
\\
0&$\cdots$&0&1%$\delta_{g-d}$
&0\\
\end{tabular},\]
where $M_{{2g}-2}$ is a $({2g}-2)\times ({2g}-2)$-matrix, $x$ is an integer, and $v_1$ and $w_1$ are column vectors with ${2g}-2$ entries. By changing basis once more, %(in fact, by changing $e_i$ to $e_i$ plus multiples of $e_{2g}$ for $i<{2g}$),
we can arrange that $x=0$ and $v_1=w_1$. % (while possibly changing $M_{2g-2}$).
The statement of Lemma~\ref{lemma:M} follows by induction on $g$, and $2d=\deg(\Delta_K)$ is immediate from the fact that calculating $\Delta_K$ using a Seifert matrix of the form~\eqref{eq:M} yields
\[\Delta_K=\det\left(M_{2d}\sqrt{t}-M_{2d}^T\frac{1}{\sqrt{t}}\right),\]
which has degree ${2d}$ since $\det(M_{2d})$ is non-zero.\end{proof}

Next, we establish that there is also a symplectic (with respect to $I$) basis for $H_1(S,\Z)$ for which the corresponding Seifert matrix is of the form~\eqref{eq:M}; in fact, this follows from a version of Witt's Theorem.
Let $B$ be a basis as provided by Lemma~\ref{lemma:M}, and let $M_{2g}$ be the corresponding Seifert matrix. %as described in~\eqref{eq:M}.
We write $H_1(S,\Z)=V_1\oplus V_2$, where $V_1$ denotes the subgroup spanned by the first $2d$ elements of $B$ and $V_2$ denotes the subgroup spanned by the other elements of $B$. Denote the lower right square of size $2g-2d$ of $M_{2g}$ by $N_{2g-2d}$. %, i.e.~$N_{2g-2d}$ corresponds the restriction of the Seifert form to the subgroup $V_2$. %which has the last $2g-2d$ elements of $B$ as basis $B_V$. %$(e_{2d+1},\cdots,e_r)$.
By~\eqref{eq:M}, we have that $M_{2g}-M_{2g}^T$ %, which represents the intersection form %on $H_1(S,\Z)$ with respect to $B$,
is
\[\left[\begin{array}{cc}M_{2d}-M_{2d}^T & 0\\0&N_{2g-2d}-N_{2g-2d}^T\end{array}\right],\] where $N_{2g-2d}-N_{2g-2d}^T$
equals the standard symplectic form on $\Z^{2g-2d}$. Since $M_{2d}-M_{2d}^T$ is invertible (which follows from $M_{2g}-M_{2g}^T$ being invertible), there is a symplectic (with respect to the restriction of $I$) basis $B_{V_1}$ for $V_1$. %, i.e.~there is a basis change that takes $M_{2d}-M_{2d}^T$ to the standard symplectic form on $\Z^{2d}$.
%since any skew-symmetric bilinear form $F$ on $\Z^{2d}$ with determinant $\pm1$ has a basis with respect to which $F$ is the standard symplectic form.
Let $B_\textrm{sympl}=(a_{1},b_{1},\cdots,a_{g},b_g)$ denote the basis for $H_1(S,\Z)$ obtained by replacing the first $2d$ elements of $B$ by $B_{V_1}$. By construction, $B_\textrm{sympl}$ is symplectic. The corresponding Seifert matrix $M_\textrm{sympl}$ is of the form~\eqref{eq:M} since $M_\textrm{sympl}$ is obtained from $M_{2g}$ by column (row) operations that involve only the first $2d$ columns (rows); in particular, $N_{2g-2d}$ remains unchanged.

Since $B_\textrm{sympl}$ is a symplectic basis for $I$, it
can be realized geometrically; i.e., for all $1\leq i\leq g$,
there exist simple closed curves $\alpha_i$ and $\beta_i$ in $S$ representing the classes $a_i$ and $b_i$, respectively, such that $\alpha_{i}$ intersects $\beta_{i}$ once transversally % for all $1\leq i\leq g$
and no other intersections occur%$\alpha_l$ and $\alpha_k$ do not intersect for all other choices of $1\leq l<k\leq r$
; see e.g.~Farb and Margalit's book~\cite[Theorem 6.4]{FarbMargalit_12_APrimerOnMCG}.
Let $L$ be any simple closed curve in $S$ separating the curves
\[K,\alpha_1,\beta_1,\cdots,\alpha_{d},\beta_d\quad\text{from}\quad
\alpha_{d+1},\beta_{d+1},\cdots,\alpha_{g},\beta_{g},\] and denote the component of $S\setminus L$ containing $\alpha_{d+1},\beta_{d+1},\cdots,\alpha_{g},\beta_{g}$ by $C$.
The existence of such a curve $L$ is evident since
\[S\setminus (K\cup\alpha_1\cup\beta_1\cup\cdots\cup\alpha_{g}\cup\beta_{g})\] is a $g+1$ punctured sphere.
%if one identifies $S$ with the a orientable surface of genus $\frac{r}{2}$ with one boundary component such that the $\alpha_i$ map to the standard symplectic curve system.
The surface $C$ has genus $g-d$ and is a Seifert surface for $L$.
Furthermore, the Seifert matrix corresponding to the basis \[([\alpha_{d+1}],[\beta_{d+1}],\cdots,[\alpha_{g}],[\beta_{g}])\] for $H_1(C,\Z)$ is $N_{2g-2d}$. %lower-right %$(r-2d)\times(r-2d)$-minor of $M_{2g}$
%is the Seifert matrix corresponding to the basis $B_{V_1}$ of $H_1(C,\Z)$. %with respect to the symplectic basis $B'$ is $N_{r-2d}$.
%Therefore, $N_{2g-2d}$ is a Seifert matrix for $L$;
Therefore, we have
\[\Delta_L=\det\left(N_{2g-2d}\sqrt{t}-N_{2g-2d}^T\frac{1}{\sqrt{t}}\right)\overset{\eqref{eq:M}}{=}1.\qedhere\]
\end{proof}
%\section{Applications and Examples}\label{sec:examples}
% We now proceed by calculating the topological slice genus for a few examples.
\begin{comment}
For the knot $12n750$, we explicitly exhibit a curve $L$ as predicted by Lemma~\ref{lemma:splitting}:
\begin{Example}
The knot $12n750$ is the closure of the $3$-braid \[a_1^2a_1a_2a_1^{-1}a_2a_1^2a_1a_2a_1^{-1}a_2,\] where $a_1$ and $a_2$ designate the standard generators of the $3$-strand braid group. In Figure, we explicitly exhibit a one holed torus $T$ in a genus $3$ Seifert surface, which is minimal since it is quasipositive, such that $L=\partial T$ has Alexander polynomial one. Indeed, $T$ is given as the neighborhood of two simple closed curves $\alpha$ and $\beta$ that intersect once and that span.
\end{Example}
\end{comment}
Lukas Lewark pointed out Lemma~\ref{lemma:M}. Originally, following the arguments in~\cite[Lemma~2]{Freedman_82_ASurgerySequenceInDimFour} and~\cite[Lemma~4.2]{GaroufalidisTeichner_04_OnKnotswithtrivialAlex}, %by Freedman
%and Garoufalidis-Teichner
%for the case of trivial Alexander polynomial,
we used changes of basis and $S$-equivalences to obtain a Seifert matrix of the form~\eqref{eq:M}, which only yields the following weaker version of Proposition~\ref{prop:splitting}. Every Seifert surface can be stabilized to a Seifert surface that contains a knot with the properties described in Proposition~\ref{prop:splitting}. We note that this version still suffices to establish Theorem~\ref{thm:alexupperboundtopslicegenus}.

The author greatly profited from the nice presentation of the Freedman's result by Garoufalidis and Teichner~\cite[Appendix]{GaroufalidisTeichner_04_OnKnotswithtrivialAlex}, where smooth $S^3$- and $B^4$-arguments are clearly separated from the application of Freedman's machinery. In fact, before discovering Proposition~\ref{prop:splitting}, which allows one to reduce Theorem~\ref{thm:alexupperboundtopslicegenus} to a single application of the fact that knots with trivial Alexander polynomial are topologically slice, our proof of Theorem~\ref{thm:alexupperboundtopslicegenus} closely followed the argument in~\cite[Appendix]{GaroufalidisTeichner_04_OnKnotswithtrivialAlex}. The following remark is related to their presentation:
\begin{Remark}\label{Rem:Obs:M}
In the case of $\Delta_K=1$, the proof of Proposition~\ref{prop:splitting} reduces to the following slight improvement of a Lemma~\cite[Lemma~4.2, first part]{GaroufalidisTeichner_04_OnKnotswithtrivialAlex} by Garoufalidis and Teichner. If a knot has trivial Alexander polynomial, then every Seifert surface has a trivial Alexander basis in the language of~\cite[Definition~4.1]{GaroufalidisTeichner_04_OnKnotswithtrivialAlex}. This follows by considering the basis $(b_1,\cdots,b_g,a_1,\cdots,a_g)$ instead of $B_\textrm{sympl}$.
%The practical relevance of the above is that, if one is looking for a knot $L$ in Seifert surface as described in Lemma~\ref{lemma:splitting}, then it is enough to start with one's favorite Seifert surface and consider bases changes of its corresponding Seifert matrix instead of base changes and $S$-equivalences.
\end{Remark}

\section{Explicit example: The knot $12n750$ and its genera}
\label{sec:12n750}
%The proof of Lemma~\ref{lemma:splitting} is constructive and can be made explicit to concretely yield a curve $L$ as described in Lemma~\ref{lemma:splitting}. and allows for concrete calculations.
For the knot $K=12n750$, which is the closure of the $3$-braid
\begin{equation}\label{eq:12n750}aaaba^{-1}baaaba^{-1}b,\quad\footnote{Here, $a$ and $b$ denote the standard generators of the 3-strand braid group corresponding to a positive crossing of the first two and the last two strands, respectively.}\end{equation} we exhibit the curve $L$ from Proposition~\ref{prop:splitting} explicitly. Let $S$ be the Seifert surface $S$ of $K$ depicted in Figure~\ref{fig:12n750}.
\begin{figure}[h]
\centering
\def\svgscale{0.5}
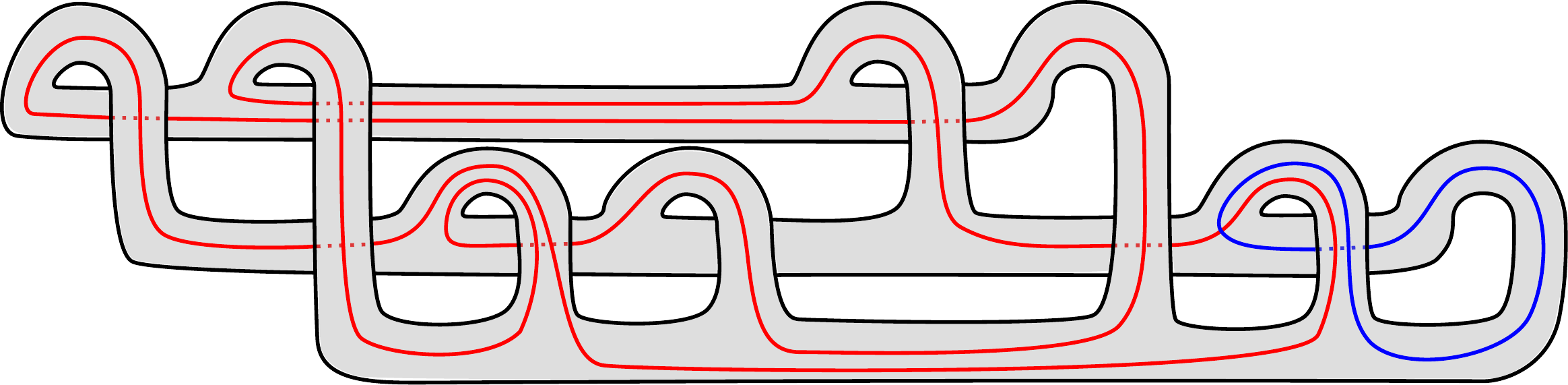
\caption{A quasipositive Seifert surface $S$ (gray) for the knot $12n750$ (black) containing two simple closed curves $\alpha$ (red) and $\beta$ (blue)
such that a neighborhood of $\alpha\cup\beta$ bounds a knot $L$ with trivial Alexander polynomial.}
\label{fig:12n750}
\end{figure}
The Seifert surface $S$ has genus $3$ and it realizes the genus and the smooth slice genus of $K$
since it is quasipositive~\cite[slice-Bennequin inequaility]{rudolph_QPasObstruction}; in fact, $S$ is the quasipositive surface canonically associated with the strongly quasipositive braid word given in~\eqref{eq:12n750}; compare~\cite{Rudolph_83_ConstructionsOfQP1,Rudolph_89_QPandnewKnotInv}. Let $\alpha$ and $\beta$ be the two once-intersecting simple closed curves depicted in Figure~\ref{fig:12n750}. A neighborhood of their union is a one-holed torus $T$ with a boundary curve $L=\partial T$ that has trivial Alexander polynomial. The latter follows since the Seifert form on $T$ with respect to the basis given by the homology classes of $\alpha$ and $\beta$ is $\left[\begin{array}{cc}
0&1\\ 0&-1
\end{array}\right]$. Now, (as in the proof of Theorem~\ref{thm:alexupperboundtopslicegenus}) one can modify $S$ by replacing $T$ by a locally-flat disc in $B^4$ to find a locally-flat surface $S^{\textrm{top}}$ of genus 2 in $B^4$ with boundary $K$. Calculating the signature of $K$ (it is $-4$) shows that $S^{\textrm{top}}$ realizes $g_4^{\rm top}(K)=2$.
\newpage
%\bibliographystyle{alpha}
%\bibliography{peterbib}
\def\cprime{$'$}

\end{document}